\documentclass[12pt, a4paper]{amsart}
\usepackage{a4wide}
\usepackage{amsmath}
\usepackage{amsthm}
\usepackage{amsfonts}
\usepackage{amssymb}
\usepackage{mathtools}
\usepackage[latin1]{inputenc}

\DeclareMathOperator{\SL}{SL}
\DeclareMathOperator{\N}{\mathbb{N}}
\DeclareMathOperator{\Z}{\mathbb{Z}}

\DeclareMathOperator{\C}{\mathbb{C}}
\DeclareMathOperator{\Q}{\mathbb{Q}}
\renewcommand{\H}{\mathbb{H}}
\DeclareMathOperator{\sgn}{sgn}
\DeclareMathOperator{\e}{\mathfrak{e}}
\DeclareMathOperator{\Mp}{Mp}
\DeclareMathOperator{\sig}{sig}
\DeclareMathOperator{\Gal}{Gal}

	\newtheorem{Satz}{Satz}[section]
	\newtheorem{Theorem}[Satz]{Theorem}
	\newtheorem{Lemma}[Satz]{Lemma}
	\newtheorem{Proposition}[Satz]{Proposition} 
	\newtheorem{Corollary}[Satz]{Corollary}
	
	\theoremstyle{definition} 
	\newtheorem{Definition}[Satz]{Definition}
	
	\newtheorem{Remark}[Satz]{Remark}

\begin{document} 

\title{Eisenstein series for the Weil representation}
\author{Markus Schwagenscheidt}
\date{}
\maketitle

	\begin{abstract}
		We compute the Fourier expansion of vector valued Eisenstein series for the Weil representation associated to an even lattice. To this end, we define certain twists by Dirichlet characters of the usual Eisenstein series associated to isotropic elements in the discriminant form of the underlying lattice. These twisted functions still form a generating system for the space of Eisenstein series but have better multiplicative properties than the individual Eisenstein series. We adapt a method of Bruinier and Kuss to obtain algebraic formulas for the Fourier coefficients of the twisted Eisenstein series in terms of special values of Dirichlet $L$-functions and representation numbers modulo prime powers of the underlying lattice. In particular, we obtain that the Fourier coefficients of the individual Eisenstein series are rational numbers. Additionally, we show that the twisted Eisenstein series are eigenforms of the Hecke operators on vector valued modular forms introduced by Bruinier and Stein.
	\end{abstract}

	\section{Introduction and statements of the main results}
	
	Eisenstein series for the Weil representation associated to an even lattice have numerous applications in number theory, for example in the theory of Borcherds products \cite{bruinierkuss, bruinierhabil, bruinierconverse}, in the classification of automorphic products of singular weight \cite{scheithauerclassification, dittmannhagemeierschwagenscheidt, scheithauerautomorphicproducts}, and in the study of automorphic Green's functions of Heegner divisors \cite{bruinierkuehn, bruiniermoeller}. In each of these examples, it is crucial to have explicit formulas for the Fourier coefficients of the Eisenstein series. Bruinier and Kuss \cite{bruinierkuss}, Scheithauer \cite{scheithauerclassification} and Kudla and Yang \cite{kudlayang} computed useful formulas for the Fourier expansion of the Eisenstein series corresponding to the zero element in the discriminant group of the underlying lattice. However, there is an Eisenstein series corresponding to each isotropic element in the discriminant group, whose coefficients have not yet been computed satisfactory. The aim of the present work is to close this gap. Let us describe our results in more detail.
	
	\subsection{Twisted Eisenstein series} Let $L$ be an even lattice of signature $(b^{+},b^{-})$ and level $N$, let $L'$ denote its dual lattice, and let $A = L'/L$ be the corresponding discriminant group. The quadratic form on $L$ and the induced finite quadratic form on $A$ will both be denoted by $Q$, and we write $(\cdot,\cdot)$ for the corresponding bilinear form. Let $\C[A]$ be the group algebra of $A$ with basis symbols $\e_{\gamma}$ for $\gamma \in A$, and let $\Mp_{2}(\Z)$ be the metaplectic double cover of $\SL_{2}(\Z)$, realized as the set of pairs $(M,\phi)$ with $M = \left(\begin{smallmatrix}a & b \\ c & d \end{smallmatrix} \right) \in \SL_{2}(\Z)$ and $\phi: \H \to \C$ holomorphic with $\phi(\tau)^{2} = c\tau + d$. The Weil representation $\rho_{A}$ associated to $A$ is a unitary representation of the metaplectic group $\Mp_{2}(\Z)$ on the group algebra $\C[A]$ (see \cite{borcherds}, Section 4). We denote the dual Weil representation by $\rho_{A}^{*}$.

		Following \cite{bruinierkuss}, for $\beta \in A$ with $Q(\beta) = 0 \pmod \Z$ and $k \in \frac{1}{2}\Z$ with $k \geq \frac{5}{2}$ we define the $\C[A]$-valued Eisenstein series
		\begin{align*}
		E_{A,\beta}(\tau) = \frac{1}{2}\sum_{(M,\phi) \in \tilde{\Gamma}_{\infty} \backslash \Mp_{2}(\Z)}\phi(\tau)^{-2k}\rho_{A}^{*}(M,\phi)^{-1}\e_{\beta},
		\end{align*}
		where $\tilde{\Gamma}_{\infty}$ is the subgroup of $\Mp_{2}(\Z)$ generated by $T = \left(\left(\begin{smallmatrix}1 & 1 \\ 0 & 1\end{smallmatrix} \right),1 \right)$. We remarkt that, due to their applications in the theory of Borcherds products, it is customary to study Eisenstein series for the dual Weil representation $\rho_{A}^{*}$ rather than for $\rho_{A}$, and we stick to this convention. The Eisenstein series converges normally and defines a holomorphic modular form of weight $k$ for $\rho_{A}^{*}$. We suppress the weight in the notation since it does not change throughout this work, and we assume that 
		\[
		\kappa = k - \frac{b^{-}}{2}+\frac{b^{+}}{2}
		\]
		is an integer since otherwise there are no non-trivial modular forms of weight $k$ for $\rho_{A}^{*}$. This means that $k$ is integral if the rank of $L$ is even and $k$ is half-integral if the rank of $L$ is odd. For $\beta = 0$ the Fourier expansion of the Eisenstein series $E_{A,0}$ has been computed in \cite{bruinierkuss} in terms of Dirichlet $L$-functions and representation numbers modulo prime powers of the lattice $L$. The starting point of the present note was the attempt to compute the Fourier expansion of $E_{A,\beta}$ for $\beta \neq 0$. Unfortunately, it turns out that the techniques of \cite{bruinierkuss} can not immediately be applied. Therefore we consider the following twisted averages of the Eisenstein series.
		
		\begin{Definition}
		Let $\beta \in A$ with $Q(\beta) = 0 \pmod \Z$, and let $N_{\beta}$ be the order of $\beta$ in $A$. For a Dirichlet character $\chi$ modulo $N_{\beta}$ we define
		\begin{align*}
		E_{A,\beta,\chi} = \sum_{\nu (N_{\beta})^{*}}\chi(\nu)E_{A,\nu\beta}.
		\end{align*}
		\end{Definition}
		
		\begin{Remark}
			\begin{enumerate}
				\item Note that $E_{A,-\beta} = (-1)^{\kappa}E_{A,\beta}$, hence $E_{A,\beta,\chi}$ vanishes identically unless $\chi(-1) = (-1)^{\kappa}$. 
				\item By summing over all Dirichlet characters modulo $N_{\beta}$ we can recover $E_{A,\beta}$. However, we will see below that the Fourier coefficients of the functions $E_{A,\beta,\chi}$ have a better multiplicative structure than the coefficients of the individual Eisenstein series $E_{A,\beta}$, which makes them easier to compute. 
				\item We were led to the twisted functions $E_{A,\beta,\chi}$ by the remarks at the end of Section~I.2 in \cite{eichlerzagier}, where similar twisted Jacobi Eisenstein series were considered. Analogous twisted Jacobi Eisenstein series of lattice index have been defined in \cite{ajouz}, but their Fourier expansion was not computed there.
			\end{enumerate}
		\end{Remark}

		\subsection{Reduction to primitive characters} We first reduce the computation of the Fourier expansion of $E_{A,\beta,\chi}$ to the case of primitive Dirichlet characters by writing $E_{A,\beta,\chi}$ as a linear combination of certain liftings of twisted Eisenstein series $E_{B_{d},\beta_{d},\psi}$ for certain finite quadratic modules $B_{d}$ and the primitive Dirichlet character $\psi$ corresponding to $\chi$. The aforementioned lifting map has been studied in \cite{scheithauermodularforms}, Section 4, and \cite{bruinierconverse}, Section 3, and is defined as follows. Let $H$ be an isotropic subgroup of $A$, i.e., $Q(\gamma) = 0 \pmod \Z$ for all $\gamma \in H$, and let $H^{\perp}$ be its orthogonal complement in $A$. Then $B = H^{\perp}/H$ is a finite quadratic module with the induced finite quadratic form. If $f = \sum_{\gamma \in B}f^{\gamma}\e_{\gamma}$ is a modular form for $\rho_{B}^{*}$, then 
	\[
	f\uparrow_{B}^{A} = \sum_{\gamma \in H^{\perp}}f^{\gamma + H}\e_{\gamma}
	\]
	is a modular form for $\rho_{A}^{*}$ of the same weight. The lifting maps cusp forms to cusp forms and preserves orthogonality to cusp forms with respect to the Petersson inner product. Of course, the Fourier expansion of $f\uparrow_{B}^{A}$ can immediately be read off from the expansion of $f$. Our first result is the following.

		\begin{Proposition}\label{proposition oldforms}
			Let $\chi$ be a Dirichlet character mod $N_{\beta}$ and let $\psi$ be the corresponding primitive Dirichlet character of modulus $N_{\psi}$. Write $N_{\beta} = N_{0}N_{0}'$ with $N_{0} = \prod_{p \mid N_{\psi}}p^{\nu_{p}(N_{\beta})}$. Then we have the formula
			\begin{align*}
			E_{A,\beta,\chi}= \psi(N_{0}')\sum_{d \mid N_{0}'}\mu(d)E_{B_{d},N_{0}'\beta,\psi}\uparrow_{B_{d}}^{A},
			\end{align*}
			where $B_{d} = H_{d}^{\perp}/H_{d}$ with the isotropic subgroup $H_{d}$ generated by $dN_{\psi}\beta$. Further, $N_{0}'\beta$ has order $N_{\psi}$ in $B_{d}$, and $B_{d} \cong L_{d}'/L_{d}$ for the even lattice $L_{d}$ generated by $L$ and $dN_{\psi}\beta$.
		\end{Proposition}
		
		The proof will be given in Section \ref{section proof proposition oldforms} below. In view of the proposition, we can think of $E_{A,\beta,\chi}$ with non-primitive $\chi$ as an \lq oldform\rq \, (see \cite{bruinierconverse}, Section 3), and assume that $\chi$ is primitive in the computation of the Fourier expansion of $E_{A,\beta,\chi}$.
		
		\subsection{The Fourier expansion of the twisted Eisenstein series} Before we can give our formula for the Fourier expansion of $E_{A,\beta,\chi}$ we need to recall some notation from \cite{bruinierkuss, bruinierkuehn}. Let
		\[
		m = b^{+}+b^{-}
		\]
		be the rank of $L$. For $\gamma \in A, n \in \Z-Q(\gamma)$, and $a \in \N$ we define a representation number modulo $a$ by
		\[
		N_{\gamma,n}(a) = \#\{r \in L/aL: Q(r-\gamma)+n \equiv 0 \pmod a\}.
		\]
		The representation number is multiplicative in $a$, i.e., $N_{\gamma,n}(a_{1}a_{2}) = N_{\gamma,n}(a_{1})N_{\gamma,n}(a_{2})$ if $(a_{1},a_{2}) = 1$. For a prime $p$ we consider the polynomial
			\[
			L_{\gamma,n}^{(p)}(X) = N_{\gamma,n}(p^{w_{p}})X^{w_{p}}+(1-p^{m-1}X)\sum_{\nu=0}^{w_{p}-1}N_{\gamma,n}(p^{\nu})X^{\nu} \in \Z[X]
			\]
			with $w_{p} = 1 + 2\nu_{p}(2N_{\beta}N_{\gamma}n)$, where $N_{\gamma}$ is the order of $\gamma$ in $A$ (note that $w_{p}$ differs slightly from the quantity defined in \cite{bruinierkuss,bruinierkuehn}). We define a discriminant $D$ corresponding to $\gamma$ and $n$ by
		\begin{align*}
		D = \begin{cases}
		(-1)^{m/2}\det(L), & \text{if $m$ is even,}\\
		2(-1)^{(m+1)/2}N_{\gamma}^{2}n\det(L), & \text{if $m$ is odd,}
		\end{cases}
		\end{align*}
		where $\det(L)$ denotes the determinant of a Gram matrix of $L$. Let $D_{0}$ be the fundamental discriminant corresponding to $D$, and let $\chi_{D_{0}}$ be the associated primitive quadratic Dirichlet character. Note that $\sgn(\det(L)) = (-1)^{b^{-}}$ which implies $\sgn(D_{0}) = (-1)^{\kappa+k}$ if $m$ is even (and $k$ is integral) and $\sgn(D_{0}) = (-1)^{\kappa + k + 1/2}$ if $m$ is odd (and $k$ is half-integral). For $\gamma \in A$ we define the greatest common divisor
		\[
		g = (N_{\beta},N_{\beta}(\gamma,\beta)).
		\]
		We split 
		\[
		N_{\beta} = N_{g}N_{g}'
		\]
		where $N_{g} = \prod_{p \mid g}p^{\nu_{p}(N_{\beta})}$, and 
		\[
		\chi = \chi_{N_{g}}\chi_{N_{g}'}
		\]
		with primitive characters mod $N_{g}$ and $N_{g}'$, respectively. Similarly, for a prime $p$ dividing $N_{g}$ we will write 
		\[
		\chi_{N_{g}} = \chi_{p}\chi_{p}'
		\]
		with primitive characters mod $p^{\nu_{p}(N_{g})}$ and $N_{g}/p^{\nu_{p}(N_{g})}$. Further, we let 
		\[ 
		G(\chi) = \sum_{u(N_{\beta})^{*}}\chi(u)e(u/N_{\beta}) \quad \text{and} \quad L(\chi,s) = \sum_{n = 1}^{\infty}\chi(n)n^{-s}
		\] 
		(with $e(x) = e^{2\pi i x}$ and $s \in \C, \Re(s) > 1$) be the Gauss sum and the Dirichlet $L$-function associated to $\chi$.
				
		The main result of this note is the following Fourier expansion of $E_{A,\beta,\chi}$, whose computation can be found in Section~\ref{section proof theorem fourier expansion}.

		\begin{Theorem}\label{theorem fourier expansion}
			Let $\chi$ be a primitive Dirichlet character modulo $N_{\beta}$ with $\chi(-1) = (-1)^{\kappa}$. The twisted Eisenstein series $E_{A,\beta,\chi}$ has the Fourier expansion
			\begin{align*}
			E_{A,\beta,\chi}(\tau) = 2\sum_{n(N_{\beta})^{*}}\chi(n)\e_{n\beta}+ \sum_{\gamma \in A}\sum_{\substack{n \in \Z - Q(\gamma) \\ n> 0}}c_{\beta,\chi}(\gamma,n)e(n\tau)\e_{\gamma}
			\end{align*}
			with Fourier coefficients of the form
			\begin{align*}
			c_{\beta,\chi}(\gamma,n) = \frac{2^{k+1}\pi^{k}n^{k-1}i^{\kappa}}{\sqrt{|A|}\Gamma(k)}\varepsilon_{\beta,\chi}(\gamma)c_{\beta,\chi}^{\text{finite}}(\gamma,n)c_{\beta,\chi}^{\text{main}}(\gamma,n)
			\end{align*}
			for $n > 0$. The constant $\varepsilon_{\beta,\chi}(\gamma)$ equals
			\begin{align*}
			\varepsilon_{\beta,\chi}(\gamma) = \frac{1}{N_{g}}\overline{\chi}_{N_{g}'}(N_{\beta}(\gamma,\beta))G(\chi).
			\end{align*}
			The main part $c_{\beta,\chi}^{\text{main}}(\gamma,n)$ is given by
			\begin{align*}
			\begin{dcases}
			\frac{1}{L(\chi\chi_{D_{0}},k)}\prod_{p \mid 2N_{\gamma}^{2}n\det(L)}\frac{L_{\gamma,n}^{(p)}\left(\chi(p)p^{1-m/2-k}\right)}{1-\chi\chi_{D_{0}}(p)p^{-k}}, & \text{if $m$ is even}, \\
			\frac{L(\chi\chi_{D_{0}},k-1/2)}{L(\chi^{2},2k-1)}\prod_{p \mid 2N_{\gamma}^{2}n\det(L)}\frac{1-\chi\chi_{D_{0}}(p)p^{1/2-k}}{1-\chi^{2}(p)p^{1-2k}}L_{\gamma,n}^{(p)}\left(\chi(p)p^{1-m/2-k}\right), & \text{if $m$ is odd}.
			\end{dcases}
			\end{align*}
			Further, the finite part is the product $c_{\beta,\chi}^{\text{finite}}(\gamma,n) = \prod_{p \mid g}c_{\beta,\chi}^{\text{finite},(p)}(\gamma,n)$ with
			\begin{align*}
			c_{\beta,\chi}^{\text{finite},(p)}(\gamma,n) &= \sum_{\alpha = \nu_{p}(g)}^{w_{p}}\chi_{p}'(p^{\alpha})p^{\alpha(1-m/2-k)}\sum_{\nu \left(p^{\nu_{p}(N_{\beta})}\right)}\!\!\!\!\!\!\sum_{\substack{u\left(p^{\nu_{p}(N_{\beta})}\right)^{*} \\ up^{\alpha} \equiv N_{\beta}(\gamma,\beta) \,\left(p^{\nu_{p}(N_{\beta})}\right)} }\!\!\!\!\!\!\overline{\chi}_{p}(u) \\
			&\quad \times\left(N_{\gamma-\nu \frac{N_{\beta}}{p^{\nu_{p}(N_{\beta})}}\beta,n+\frac{p^{\alpha}\nu u}{p^{\nu_{p}(N_{\beta})}}}(p^{\alpha})-p^{m-1}N_{\gamma-\nu \frac{N_{\beta}}{p^{\nu_{p}(N_{\beta})}}\beta,n+\frac{p^{\alpha}\nu u}{p^{\nu_{p}(N_{\beta})}}}(p^{\alpha-1})\right).
			\end{align*}
		\end{Theorem}
		
		\begin{Remark}
			\begin{enumerate}
				\item Up to the appearance of the character $\chi$ in some places, the formula for the main part of $E_{A,\beta,\chi}$ resembles the formula for $E_{A,0}$ given in \cite{bruinierkuss}, Theorem~7.
				\item In \cite{bruinierkuehn} the Fourier expansion of a real-analytic analog $E_{A,\beta}(\tau,s)$, $s \in \C$, of the Eisenstein series $E_{A,\beta}(\tau)$ has been computed for $\beta = 0$, using the same techniques as in \cite{bruinierkuss}. We remark that our results can easily be generalized to real-analytic twisted Eisenstein series $E_{A,\beta,\chi}(\tau,s)$. It could be interesting to study their special value at $s = 0$ for low weights as in \cite{williams}.
		
				\item In her Master's thesis \cite{masterklimmek}, Klimmek computed a different formula for $E_{A,\beta,\chi}$ in the case that the rank of $L$ is even and the level of $L$ is square free. Following an idea of Scheithauer (see \cite{scheithauerclassification}, Section 7), she wrote $E_{A,\beta,\chi}$ as a certain lift of a scalar valued Eisenstein series for a congruence subgroup, whose Fourier coefficients are well known. The formulas obtained in this way involve special values of Dirichlet $L$-functions, twisted divisor sums and elementary invariants of the underlying lattice, but no representation numbers modulo prime powers.
				\item Recently, the Fourier coefficients of Jacobi Eisenstein series of lattice index (and corresponding to the zero element in the discriminant form) have been studied in \cite{woitalla,mocanu}. Using the techniques of the present work it should be possible to compute the expansions of all Jacobi Eisenstein series of lattice index.
			\end{enumerate}
		\end{Remark}
		
		Using the functional equation of the Dirichlet $L$-function and its evaluation at negative integers in terms of Bernoulli polynomials we may infer the following rationality result from Proposition~\ref{proposition oldforms} and Theorem~\ref{theorem fourier expansion}.
		
		\begin{Corollary}\label{corollary}
			For every Dirichlet character $\chi$ modulo $N_{\beta}$ the coefficients of $E_{A,\beta,\chi}$ lie in $\Q(\chi) \subset \Q\left(\zeta_{\varphi(N_{\beta})}\right)$, where $\varphi$ is Euler's totient function and $\zeta_{\varphi(N_{\beta})}$ is a primitive $\varphi(N_{\beta})$-th root of unity. The Galois group of $\Q(\zeta_{\varphi(N_{\beta})})$ acts on the Fourier coefficients of $E_{A,\beta,\chi}$ by 
			\[
			\sigma\big(c_{\beta,\chi}(\gamma,n)\big) = c_{\beta,\sigma\circ\chi}(\gamma,n),
			\]
			for $\sigma \in \Gal\left(\Q\left(\zeta_{\varphi(N_{\beta})}\right)/\Q\right)$.		 In particular, the coefficients of the individual Eisenstein series $E_{A,\beta} = \sum_{\chi}E_{A,\beta,\chi}$ are rational numbers. 
		\end{Corollary}
		
		For the proof we refer to Section \ref{section proof corollary}.
		
		\subsection{The action of Hecke operators on twisted Eisenstein series} In order to emphasize the significance of the twisted Eisenstein series $E_{A,\beta,\chi}$ we finally show that they are eigenforms under the Hecke operators on vector valued modular forms introduced by Bruinier and Stein in \cite{bruinierstein}. Let $p$ be a prime which is coprime to the level $N$ of $L$. If the rank $m$ of $L$ is even, we assume that $p \equiv r^{2} \pmod N$ for some $r \in \Z/N\Z$, and consider the Hecke operator 
		\[
		T_{r}(p) = T\left(\begin{pmatrix}p & 0 \\ 0 & 1 \end{pmatrix},r \right)
		\]
		as in \cite{bruinierstein}, Theorem~4.2. If $m$ is odd, we consider the Hecke operator 
		\[
		T(p^{2}) = T\left(\begin{pmatrix}p^{2} & 0 \\ 0 & 1 \end{pmatrix},1,p,1 \right)
		\] as in \cite{bruinierstein}, Theorem~4.10. The Hecke operators map cusp forms to cusp forms and they are self-adjoint with respect to the Petersson inner product. Their action on the Fourier expansion of a vector valued modular form for $\rho_{A}^{*}$ can explicitly be computed, compare \cite{bruinierstein}, Theorem~4.2 and Theorem~4.10, or Section~\ref{section proof proposition hecke} below.

		\begin{Proposition}\label{proposition hecke}
			Let $\chi$ be a (not necessarily primitive) Dirichlet character modulo $N_{\beta}$. Let $p$ be a prime which is coprime to the level of $L$. The Eisenstein series $E_{A,\beta,\chi}$ is an eigenform of $T_{r}(p)$ if $m$ is even or $T(p^{2})$ if $m$ is odd with eigenvalue
			\begin{align*}
			\begin{dcases}
			\chi(r) +  p^{k-1}\overline{\chi}(r), & \text{if $m$ is even,} \\
			\chi(p) + p^{2k-2}\overline{\chi}(p), & \text{if $m$ is odd.}
			\end{dcases}
			\end{align*}
		\end{Proposition}
		
		The statement follows quite easily from the results of \cite{bruinierstein}, and will be proved in Section~\ref{section proof proposition hecke}. We remark that the individual Eisenstein series $E_{A,\beta}$ are in general not Hecke eigenforms, compare Lemma~\ref{lemma hecke} below. The action of Hecke operators on similar twisted Jacobi Eisenstein series has been computed in \cite{ajouz}.

		\subsection*{Acknowledgements}
		
		I cordially thank Brandon Williams for enlightening discussions and his numerical calculations to validate the formulas above. I also thank Jan Bruinier, Johannes Buck, Franziska Klimmek, Andreea Mocanu and Sebastian Opitz for helpful discussions. The author was partially supported by DFG grant BR-2163/4-1 and the LOEWE-Schwerpunkt USAG.
		
	\section{Proofs of the main results}	
	
	\subsection{Proof of Proposition \ref{proposition oldforms}}\label{section proof proposition oldforms}
	
	We start with a simple lemma.
	
	\begin{Lemma}\label{lemma lifting}
		Let $H$ be an isotropic subgroup of $A$ and $B = H^{\perp}/H$. Let $\beta \in B$ with $Q(\beta) = 0 \pmod \Z$. Then we have
		\[
		E_{B,\beta}\uparrow_{B}^{A} = \sum_{\gamma \in H}E_{A,\beta+\gamma}.
		\]
	\end{Lemma}
	
	\begin{proof}
		It is well known that the Eisenstein series $E_{B,\beta}$ is the unique modular form of weight $k$ for $\rho_{B}^{*}$ which has constant term $\e_{\beta} + (-1)^{\kappa}\e_{-\beta}$ and is orthogonal to cusp forms with respect to the Petersson inner product. Since $\uparrow_{B}^{A}$ preserves orthogonality to cusp forms, it suffices to compare the constant terms of the two functions above. It is easy to check that the constant term of both functions is given by $\sum_{\gamma \in \beta+H}(\e_{\gamma}+(-1)^{\kappa}\e_{-\gamma})$.
	\end{proof}
	
		Now we proceed to the proof of Proposition \ref{proposition oldforms}. Let $\psi$ be the primitive character mod $N_{\psi}$ associated to $\chi$. Split $N_{\beta} = N_{0}N_{0}'$, where $N_{0}$ is the exact divisor of $N_{\beta}$ having the same prime divisors as $N_{\psi}$. Note that $(N_{0},N_{0}') = 1$. Thus we can write $n = n_{1}N_{0} + n_{2}N_{0}'$ with $n_{1}$ running mod $(N_{0}')^{*}$ and $n_{2}$ running mod $(N_{0})^{*}$. Since $N_{0}$ has the same prime divisors as $N_{\psi}$, we can further write $n_{2} = a + N_{\psi}b$ with $a$ running mod $(N_{\psi})^{*}$ and $b$ running mod $N_{0}/N_{\psi}$. We obtain
		\[
		E_{A,\beta,\chi} = \psi(N_{0}')\sum_{a(N_{\psi})^{*}}\psi(a)\sum_{b (N_{0}/N_{\psi})}\sum_{n_{1}(N_{0}')^{*}}E_{A,(n_{1}N_{0}+bN_{\psi}N_{0}'+aN_{0}')\beta}.
		\]
		We have the formula 
		\[
		\sum_{n(N)^{*}}f(n) = \sum_{d \mid N}\mu(d)\sum_{m(N/d)}f(dm)
		\]
		for any natural number $N$ and any function $f$ on $\Z/N\Z$. Thus we obtain
		\begin{align*}
		E_{A,\beta,\chi} &= \psi(N_{0}')\sum_{a(N_{\psi})^{*}}\psi(a)\sum_{b (N_{0}/N_{\psi})}\sum_{d \mid N_{0}'}\mu(d)\sum_{n(N_{0}'/d)}E_{A,(ndN_{0}+bN_{\psi}N_{0}'+aN_{0}')\beta}.
		\end{align*}
		We write
		\[
		(ndN_{0} + bN_{\psi}N_{0}' + aN_{0}')\beta = (nN_{0}/N_{\psi} + bN_{0}'/d)dN_{\psi}\beta + aN_{0}'\beta.
		\]
		If $n$ runs mod $N_{0}'/d$ and $b$ mod $N_{0}/N_{\psi}$, then $nN_{0}/N_{\psi} + bN_{0}'/d$ runs mod $N_{\beta}/dN_{\psi}$. Thus we get
		\[
		E_{A,\beta,\chi} = \psi(N_{0}')\sum_{d \mid N_{0}'}\mu(d)\sum_{a(N_{\psi})^{*}}\psi(a)\sum_{m(N_{\beta}/dN_{\psi})}E_{A,mdN_{\psi}\beta+aN_{0}'\beta}.
		\]
		Now let $H_{d}$ be the subgroup of $A$ generated by $dN_{\psi}\beta$. It is an isotropic subgroup of order $N_{\beta}/dN_{\psi}$. Further, $aN_{0}'\beta \in H_{d}^{\perp}$ is isotropic and has order $N_{\psi}$ in $B_{d} = H_{d}^{\perp}/H_{d}$. Using Lemma \ref{lemma lifting} we find
		\[
		\sum_{m(N_{\beta}/dN_{\psi})}E_{A,mdN_{\psi}\beta+aN_{0}'\beta} = E_{B_{d},aN_{0}'\beta}\uparrow_{B_{d}}^{A}.
		\]
		This gives the formula in Proposition \ref{proposition oldforms}.
		
		Since $Q(\beta) \equiv 0 \pmod \Z$, we see that the lattice $L_{d}$ generated by $L$ and $dN_{\psi}\beta$ is again even. Its dual lattice is given by
		\[
		L_{d}' = \{x \in L' : (x,dN_{\psi}\beta) \in \Z\}.
		\]
		It is clear that $L_{d}/L \cong H_{d}$ and $L_{d}'/L \cong H_{d}^{\perp}$, which implies
		\[
		B_{d} = H_{d}^{\perp}/H_{d} \cong (L_{d}'/L)/(L_{d}/L) \cong L_{d}'/L_{d}.
		\]
		This finishes the proof of Proposition \ref{proposition oldforms}.
		
		\subsection{Proof of Theorem \ref{theorem fourier expansion}}\label{section proof theorem fourier expansion}
		
		First, a standard computation yields the following Fourier expansion of the individual Eisenstein series $E_{A,\beta}$.
		
		\begin{Proposition}[\cite{bruinierhabil}, Theorem 1.6]
		The Eisenstein serie $E_{A,\beta}$ has the Fourier expansion
		\[
		E_{A,\beta}(\tau) = \e_{\beta} + (-1)^{\kappa}\e_{-\beta} + \sum_{\gamma \in A}\sum_{\substack{n \in \Z - Q(\gamma) \\ n > 0}}c_{\beta}(\gamma,n)e(n\tau)\e_{\gamma}
		\]
		with
		\begin{align*}
		c_{\beta}(\gamma,n) = \frac{(2\pi)^{k}n^{k-1}}{\Gamma(k)}\sum_{c \in \Z \setminus \{0\}}|c|^{1-k}H_{c}^{*}(\beta,0,\gamma,n)
		\end{align*}
		for $n > 0$, where
		\[
		H_{c}^{*}(\beta,m,\gamma,n) = \frac{e^{-\pi i \sgn(c)k/2}}{|c|}\sum_{d(c)^{*} }\rho_{\beta\gamma}\widetilde{\begin{pmatrix}a & b \\ c & d \end{pmatrix}}e\left(\frac{ma+nd}{c} \right)
		\]
		is a Kloosterman sum. Here $a,b \in \Z$ are such that $\widetilde{\left(\begin{smallmatrix}a & b \\ c & d \end{smallmatrix} \right)} = \left(\left(\begin{smallmatrix}a & b \\ c & d \end{smallmatrix}\right),\sqrt{c\tau + d} \right)\in \Mp_{2}(\Z)$, and $\rho_{\beta\gamma}\widetilde{\left(\begin{smallmatrix}a & b \\ c & d \end{smallmatrix}\right)}$ denotes a coefficient of the Weil representation as in \cite{bruinierkuss}, Section~3.
		\end{Proposition} 
	
		Multiplying by $\chi(\nu)$ and summing up over $\nu \in (\Z/N_{\beta}\Z)^{*}$, we obtain a first Fourier expansion of $E_{A,\beta,\chi}$. However, it is not very satisfying since the coefficients are given by infinite series involving complicated Kloosterman sums. Using Shintani's formula for the coefficients of the Weil representation (see \cite{shintani}, Proposition~1.6, or \cite{bruinierkuss}, Proposition~1) we further compute
		\begin{align*}
		\sum_{\nu(N_{\beta})^{*}}\chi(\nu)\sum_{c \neq 0}|c|^{1-k}H_{c}^{*}(\nu\beta,0,\gamma,n) = \frac{2i^{-\kappa}}{\sqrt{|A|}}\sum_{c\geq 1 }c^{-k-m/2}G_{\gamma,n}(c;\beta,\chi),
		\end{align*}
		where
		\begin{align*}
		G_{\gamma,n}(c;\beta,\chi) = \sum_{\nu (N_{\beta})^{*}}\chi(\nu) \sum_{d(c)^{*}}\sum_{r \in L/cL}e\left(\frac{aQ(\nu\beta+r)-(\gamma,\nu\beta+r)+d(Q(\gamma)+n)}{c}\right).
		\end{align*}
		The sum $G_{\gamma,n}(c;\beta,\chi)$ is multiplicative in $c$ in the following sense.

		\begin{Lemma}
			Let $c = c_{1}c_{2}$ with $(c_{1},c_{2}) = 1$, and let $N_{1} = \prod_{p \mid c_{1}}p^{\nu_{p}(N_{\beta})}$ and $N_{2} = N_{\beta}/N_{1}$. Write $\chi = \chi_{1}\chi_{2}$ with characters $\chi_{1}$ mod $N_{1}$ and $\chi_{2}$ mod $N_{2}$. Then we have the formula
			\begin{align*}
			G_{\gamma,n}(c;\beta,\chi) = \chi_{1}(N_{2}c_{2})\chi_{2}(N_{1}c_{1})G_{\gamma,n}(c_{1};N_{2}\beta,\chi_{1})G_{\gamma,n}(c_{2};N_{1}\beta,\chi_{2}).
			\end{align*}
		\end{Lemma}
	
		\begin{proof}
			We split $\nu = N_{2}c_{2}\nu_{1} + N_{1}c_{1}\nu_{2}$ with $\nu_{i} \in (\Z/N_{i}\Z)^{*}$ and $d = c_{2}d_{1} + c_{1}d_{2}$ with $d_{i} \in (\Z/c_{i}\Z)^{*}$. Then we can take $a = \overline{c}_{2}a_{1} + \overline{c}_{1}a_{2}$ with $a_{i} \in \Z$ such that $a_{i}d_{i} \equiv 1 \pmod {c_{i}}$ and $\overline{c}_{i} \in \Z$ such that $c_{i}\overline{c}_{i} \equiv 1 \pmod {c/c_{i}}$ and $\overline{c}_{i} \equiv 0 \pmod{c_{i}}$. Further, we write $r = c_{2}r_{1} + c_{1}r_{2}$ with $r_{i} \in L/c_{i}L$. If we plug this into the definition of $G(c;\beta,\chi)$ and use $\chi(N_{2}c_{2}\nu_{1} + N_{1}c_{1}\nu_{2}) = \chi_{1}(N_{2}c_{2}\nu_{1})\chi_{2}(N_{1}c_{1}\nu_{2})$, we obtain the stated formula.
		\end{proof}
		
		By the above lemma we have
		\begin{align}
		\begin{split}\label{eq splitting}
		&\sum_{c \geq 1 }c^{-k-m/2}G_{\gamma,n}(c;\beta,\chi) = \chi_{N_{g}}(N_{g}')\chi_{N_{g}'}(N_{g})\sum_{\substack{c \geq 1 \\ (c,g) = 1}}\chi_{N_{g}}(c)c^{-k-m/2}G_{\gamma,n}\left(c;N_{g}\beta,\chi_{N_{g}'}\right) \\
		&\qquad  \times \prod_{p \mid g}\left(\sum_{\alpha = 0}^{\infty}\chi_{p}'(p^{\nu_{p}(N_{g})}p^{\alpha})p^{-\alpha(k+m/2)}G_{\gamma,n}\left(p^{\alpha},\frac{N_{\beta}}{p^{\nu_{p}(N_{\beta})}}\beta,\chi_{p}\right)\right).
		\end{split}
		\end{align}
		Recall that here $\chi = \chi_{N_{g}}\chi_{N_{g}'}$, and $\chi_{N_{g}} = \chi_{p}\chi_{p}'$ where $\chi_{p}$ has conductor $p^{\nu_{p}(N_{g})}$ and $\chi_{p}'$ has conductor $\frac{N_{g}}{p^{\nu_{p}(N_{g})}}$. Note that $N_{g}\beta$ has order $N_{g}'$ in $A$ which matches the conductor of $\chi_{N_{g}'}$. Similarly, the order of $\frac{N_{\beta}}{p^{\nu_{p}(N_{\beta})}}\beta$ in $A$ equals $p^{\nu_{p}(N_{\beta})}$, which is the conductor of $\chi_{p}$. Next, we compute $G_{\gamma,n}(c;\beta,\chi)$ in the relevant cases.
	
		\begin{Proposition}\label{proposition main technical proposition}
			Let $\beta \in A$ with $Q(\beta) = 0 \pmod \Z$, let $N_{\beta}$ be the order of $\beta$ in $A$, and let $\chi$ be a primitive Dirichlet character mod $N_{\beta}$. Let $g = (N_{\beta},N_{\beta}(\gamma,\beta))$.
			\begin{enumerate}
				\item Suppose that $(c,g) = 1$. Then we have the formula
				\begin{align*}
				G_{\gamma,n}(c;\beta,\chi) = \chi(c)\overline{\chi}(-N_{\beta}(\gamma,\beta))G(\chi)c^{m}\sum_{a \mid c}\mu(c/a)a^{1-m}N_{\gamma,n}(a).
				\end{align*}
				\item Let $p$ be a prime dividing $g$, and suppose that $N_{\beta} = p^{\nu_{p}(N_{\beta})}$. Then we have $G_{\gamma,n}(p^{\alpha};\beta,\chi) = 0$ for $0 \leq \alpha < \nu_{p}(g)$, and
				\begin{align*}
				&G_{\gamma,n}(p^{\alpha};\beta,\chi) = \chi(-1)G(\chi)p^{\alpha-\nu_{p}(N_{\beta})} \!\!\!\!\!\!\!\!\!\!\sum_{\substack{u\left(p^{\nu_{p}(N_{\beta})}\right)^{*} \\ up^{\alpha} \equiv p^{\nu_{p}(N_{\beta})}(\gamma,\beta)\ \left(p^{\nu_{p}(N_{\beta})}\right)}}\!\!\!\!\!\!\!\!\!\!\overline{\chi}(u) \\
			& \qquad\qquad  \times \sum_{\nu \left(p^{\nu_{p}(N_{\beta})}\right)}\left(N_{\gamma-\nu\beta,n+\frac{p^{\alpha}\nu u}{p^{\nu_{p}(N_{\beta})}}}(p^{\alpha})-p^{m-1}N_{\gamma-\nu\beta,n+\frac{p^{\alpha}\nu u}{p^{\nu_{p}(N_{\beta})}}}(p^{\alpha-1})\right)
				\end{align*}
				for $\alpha \geq \nu_{p}(g)$.
			\end{enumerate}
		\end{Proposition}
		
		\begin{proof}
		Suppose that $(c,g) = 1$. We first write
		\begin{align*}
		G_{\gamma,n}(c;\beta,\chi) &= \frac{1}{c}\sum_{\nu (N_{\beta}c)}\chi(\nu) \sum_{d(c)^{*}}\sum_{r \in L/cL}e\left(\frac{aQ(\nu\beta+r)-(\gamma,\nu\beta+r)+d(Q(\gamma)+n)}{c}\right).
		\end{align*}
		Since $\chi$ is primitive, we have
		\[
		\chi(\nu) = \frac{1}{G(\overline{\chi})}\sum_{u(N_{\beta})^{*}}\overline{\chi}(u)e\left(\frac{u\nu}{N_{\beta}}\right), \qquad G(\overline{\chi}) = \sum_{u(N_{\beta})^{*}}\overline{\chi}(u)e\left(\frac{u}{N_{\beta}}\right),
		\]
		for any $\nu$, not necessarily coprime to $N_{\beta}$. We plug this into the above formula for $G_{\gamma,n}(c;\beta,\chi)$ and consider the sum over all terms containing $\nu$,
		\begin{align*}
		\sum_{\nu (N_{\beta}c)}e\left( \frac{\nu^{2}N_{\beta}aQ(\beta)+ \nu(aN_{\beta}(\beta,r)-N_{\beta}(\gamma,\beta)+uc)}{N_{\beta}c}\right).
		\end{align*}
		This is a quadratic Gauss sum, which vanishes unless $uc \equiv N_{\beta}(\gamma,\beta)\pmod {N_{\beta}}$. Since we assume $(c,g) = 1$, this congruence can only be satisfied by some $u \in (\Z/N_{\beta}\Z)^{*}$ if $(c,N_{\beta}) = 1$, that is, $G_{\gamma,n}(c;\beta,\chi) = 0$ if $(c,N_{\beta}) > 1$. In the case $(c,N_{\beta}) = 1$ we have $u \equiv \overline{c}N_{\beta}(\gamma,\beta) \pmod {N_{\beta}}$, where $c\overline{c} \equiv 1 \pmod {N_{\beta}}$, and we obtain
		\begin{align*}
		&G_{\gamma,n}(c;\beta,\chi) = \frac{\chi(c)\overline{\chi}(N_{\beta}(\gamma,\beta))}{G(\overline{\chi})}\sum_{\nu (N_{\beta})}  \\
		&\qquad\times\sum_{d(c)^{*}}\sum_{r \in L/cL}e\left(\frac{aQ(\nu\beta+r)-(\gamma,\nu\beta+r)+d(Q(\gamma)+n)+ c\overline{c}(\gamma,\nu\beta)}{c}\right).
		\end{align*}
		Next, we replace $r$ by $r - \nu N_{\beta}\overline{N}_{\beta} \beta$ where $N_{\beta}\overline{N}_{\beta} \equiv 1 \pmod {c}$, which makes sense since $N_{\beta}\beta \in L$. It is then easy to check that the numerator of the resulting expression in the exponential above is equal to $aQ(r)-(\gamma,r) + d(Q(\gamma) + n)$ modulo $c$. Thus we arrive at
		\begin{align*}
		G_{\gamma,n}(c;\beta,\chi) &= \frac{\chi(c)\overline{\chi}(N_{\beta}(\gamma,\beta))}{G(\overline{\chi})}N_{\beta}\sum_{d(c)^{*}}\sum_{r \in L/cL}e\left(\frac{aQ(r)-(\gamma,r)+d(Q(\gamma)+n)}{c}\right).
		\end{align*}
		As in the proof of \cite{bruinierkuss}, Proposition 3, we can now replace $r$ by $dr$ and use the standard evaluation of the Ramanujan sum $\sum_{c(d)^{*}}e(nd/c) = \sum_{a \mid (n,c)}\mu(c/a)a$ in terms of the Moebius function to obtain
		\begin{align*}
		\sum_{d(c)^{*}}\sum_{r \in L/cL}e\left(\frac{d(Q(r-\gamma)+n)}{c}\right) &= \sum_{r \in L/cL}\sum_{a \mid (c,Q(r-\gamma)+n)}\mu(c/a)a \\
		&= \sum_{a \mid c}\mu(c/a)a(c/a)^{m}N_{\gamma,n}(a).
		\end{align*}
		Finally, we use that $\chi(-1)G(\chi)G(\overline{\chi}) = |G(\chi)|^{2} = N_{\beta}$ since $\chi$ is primitive mod $N_{\beta}$. This yields the stated formula in the case $(c,g) = 1$. Note that the formula also gives the correct result if $c$ is not coprime to $N_{\beta}$ since in this case $\chi(c) = 0$.
		
		Let us now suppose that $c = p^{\alpha}$ for some prime $p$ dividing $g$, and that $N_{\beta} = p^{\nu_{p}(N_{\beta})}$. By the same arguments as above we can write
		\begin{align*}
		&G_{\gamma,n}(p^{\alpha};\beta,\chi) = \frac{1}{G(\overline{\chi})}\sum_{\nu \left(p^{\nu_{p}(N_{\beta})}\right)}\sum_{\substack{u \left(p^{\nu_{p}(N_{\beta})}\right)^{*} \\ up^{\alpha}\equiv p^{\nu_{p}(N_{\beta})}(\gamma,\beta) \, \left(p^{\nu_{p}(N_{\beta})}\right)}}\overline{\chi}(u)  \\
		&\qquad \times\sum_{d(p^{\alpha})^{*}}\sum_{r \in L/p^{\alpha}L}e\left(\frac{aQ(\nu\beta+r)-(\gamma,\nu\beta+r)+d(Q(\gamma)+n)+p^{\alpha-\nu_{p}(N_{\beta})}\nu u }{p^{\alpha}}\right).
		\end{align*} 
		The equation $up^{\alpha} \equiv p^{\nu_{p}(N_{\beta})}(\gamma,\beta) \pmod {p^{\nu_{p}(N_{\beta})}}$ can only be satisfied for some $u \in (\Z/p^{\nu_{p}(N_{\beta})}\Z)^{*}$ if $(p^{\alpha},p^{\nu_{p}(N_{\beta})}) = p^{\nu_{p}(g)}$, which implies $G_{\gamma,n}(p^{\alpha};\beta,\chi) = 0$ for $0 \leq \alpha < p^{\nu_{p}(g)}$. Since $d \in (\Z/p^{\alpha}\Z)^{*}$ is coprime to $p^{\nu_{p}(N_{\beta})}$ we can replace $r$ by $dr$ and $\nu$ by $d \nu$ to rewrite the second line of the last formula above as
		\begin{align*}
		\sum_{d(p^{\alpha})^{*}}\sum_{r \in L/p^{\alpha}L}e\left(\frac{d(Q(\nu\beta+r-\gamma)+n+p^{\alpha-\nu_{p}(N_{\beta})}\nu u) }{p^{\alpha}}\right).
		\end{align*}
		Note that the numerator in the exponential function is an integer under the conditions on $u$. We evaluate the Ramanujan sum as before to obtain the stated formula.		
		\end{proof}

		Now we plug these formula into equation \eqref{eq splitting}. The first line on the right-hand side of equation \eqref{eq splitting} becomes
		\begin{align*}
		\chi_{N_{g}}(N_{g}')\chi_{N_{g}'}(N_{g})\overline{\chi}_{N_{g}'}(-N_{\beta}(\gamma,\beta))G(\chi_{N_{g}'})\sum_{\substack{c \geq 1 \\ (c,g) = 1}}\chi(c)c^{m/2-k}\sum_{a \mid c}\mu(c/a)a^{1-m}N_{\gamma,n}(a).
		\end{align*}
		The series over $c$ can be computed in the same way as it was done for $\chi = 1$ in \cite{bruinierkuss}, Section~4, or \cite{bruinierkuehn}, Section~3. Therefore we leave the details of the computation to the reader. We remark that the most important ingredient is the explicit evaluation of the representation numbers $N_{\gamma,n}(a)$ due to Siegel \cite{siegel} (see also \cite{bruinierkuss}, Theorem~6). We obtain that the series over $c$ equals the main part $c_{\beta,\chi}^{\text{main}}(\gamma,n)$ as stated in Theorem~\ref{theorem fourier expansion}.

		In order to simplify the expression in the second line of equation \eqref{eq splitting}, we use the following lemma, which is due to Siegel (\cite{siegel}, Hilfssatz~13) but is given in a more convenient form in \cite{bruinierkuss}, Lemma~5.
		
		\begin{Lemma}[\cite{bruinierkuss}, Lemma~5]
			Let $\gamma \in A$ and $n \in \Z - Q(\gamma)$. Let $N_{\gamma}$ be the order of $\gamma$ in $A$ and let $p$ be a prime. For $\alpha > 1 + 2\nu_{p}(2N_{\gamma}n)$ we have
			\[
			N_{\gamma,n}(p^{\alpha+1}) = p^{m-1}N_{\gamma,n}(p^{\alpha}).
			\]
		\end{Lemma}
		
		The lemma implies that $G_{\gamma,n}(p^{\alpha};\beta,\chi)$ for $p \mid g$ and $N_{\beta} = p^{\nu_{p}(N_{\beta})}$ vanishes for $\alpha > w_{p}$, so the series in the second line of \eqref{eq splitting} is actually a finite sum.
		
		Finally, we use the multiplicativity of Gauss sums to rewrite
		\[
		\chi_{N_{g}}(N_{g}')\chi_{N_{g}'}(N_{g})G(\chi_{N_{g}'})\prod_{p \mid g}\chi_{p}'(p^{\nu_{p}(N_{g})})G(\chi_{p}) = G(\chi).
		\]
		Taking everything together, we obtain the formula given in Theorem \ref{theorem fourier expansion}.
		
	\subsection{Proof of Corollary \ref{corollary}}\label{section proof corollary}
	
	It is clear that the finite part and the finite products occuring in the main part of the Fourier coefficients of $E_{A,\beta,\chi}$ lie in $\Q(\chi)$. Hence we need to investigate the $L$-factors appearing in the main part more closely. Recall that, if a character $\xi$ modulo $f$ is induced from a character $\xi_{0}$ modulo $f_{0}$, then their Dirichlet $L$-functions are related by
	\[
	L(\xi,s) = L(\xi_{0},s)\prod_{p \mid f}(1-\xi_{0}(p)p^{-s}),
	\]
	and $L(\xi_{0},s)$ satsfies the functional equation
	\[
	L(\xi_{0},s) = \frac{f_{0}^{-s}\Gamma(\frac{1-s-\delta}{2})G(\xi_{0})}{\pi^{1/2-s}\Gamma(\frac{s+\delta}{2})i^{\delta}}L(\overline{\xi}_{0},1-s).
	\]
	where $\delta = 0$ if $\xi_{0}$ is even and $\delta = 1$ if $\xi_{0}$ is odd. Further, if $s$ is a positive integer with $s \equiv \delta \pmod 2$ we have the evaluation
	\[
	L(\overline{\xi}_{0},1-s) = -\frac{f_{0}^{s-1}}{s}\sum_{n(f_{0})}\overline{\xi}_{0}(n)B_{s}(n/f_{0})
	\]
	with the usual Bernoulli polynomials $B_{s}(x) \in \Q[x]$ (see \cite{zagier}, Chapter~I.7). If we apply this to $s = k$ and the characters $\chi\chi_{D_{0}}$ and $\chi^{2}$, we see that all powers of $\pi$ and all gamma factors appearing in the Fourier coefficients of $E_{A,\beta,\chi}$ combine to rational numbers. Next, we treat the occuring Gauss sums. Let us assume for simplicity that $\chi$ and $\chi_{D_{0}}$ have coprime conductors, and that $\chi$ and $\chi^{2}$ are primitive. The general case is similar, but more technical. Under these conditions, we have the well-known relations
	\[
	G(\chi\chi_{D_{0}}) = \chi(|D_{0}|)\chi_{D_{0}}(N_{\beta})G(\chi)G(\chi_{D_{0}}), \quad G(\chi^{2}) = \frac{G(\chi)^{2}}{J(\chi,\chi)},
	\]
	where $J(\chi,\chi) = \sum_{a (N_{\beta})} \chi(a)\chi(1-a) \in \Q(\chi)$ is a Jacobi sum, and the evaluation 
	\[
	G(\chi_{D_{0}}) = i^{a}\sqrt{|D_{0}|},
	\]
	where $a = 0$ if $D_{0} > 0$ and $a = 1$ if $D_{0} < 0$. Recall that $\sgn(D_{0}) = (-1)^{\kappa+k}$ if the rank of $L$ is even and $\sgn(D_{0}) = (-1)^{\kappa+k+1/2}$ if the rank of $L$ is odd. All occuring Gauss sums $G(\chi)$ cancel out, $\sqrt{|D_{0}|}$ combines with $\sqrt{|A|} = \sqrt{|\det(L)|}$ (and $(2n)^{k}$ if the rank is odd) to a rational number, and the sum of all powers at $i$ is even. Hence the coefficients $c_{\beta,\chi}(\gamma,n)$ of $E_{A,\beta,\chi}$ lie in $\Q(\chi)$. Further, from the resulting expression it is easy to see that $\sigma(c_{\beta,\chi}(\gamma,n)) = c_{\beta,\sigma\circ \chi}(\gamma,n)$ for $\sigma \in \Gal(\Q(\zeta_{\varphi(N_{\beta})})/\Q)$. The Galois group of $\Q(\zeta_{\varphi(N_{\beta})})$ only permutes the summands in $c_{\beta}(\gamma,n)=\sum_{\chi}c_{\beta,\chi}(\gamma,n)$. Therefore, the coefficients of $E_{A,\beta}$ are invariant under this Galois group, hence rational. 
	 
	\subsection{Proof of Proposition \ref{proposition hecke}}\label{section proof proposition hecke}
	
	Let $\langle \cdot,\cdot \rangle$ denote the usual Petersson inner product on modular forms of weight $k$ for $\rho_{A}^{*}$. If $m$ is even, then the Hecke operator $T_{r}(p)$ satisfies
	\[
	\langle f|T_{r}(p),g \rangle = \langle f,g|T_{r}(p)\rangle
	\]
	for all holomorphic modular forms $f,g \in M_{k,\rho_{A}^{*}}$, if the inner products exist. If $m$ is odd, the analogous relation holds to $T(p^{2})$. Since the Hecke operators also map cusp forms to cusp forms, we see that they preserve orthogonality to cusp forms.
	
	By \cite{bruinierstein}, Theorem~4.2 and Theorem~4.10, the Hecke operatos $T_{r}(p)$ (if $m$ is even) or $T(p^{2})$ (if $m$ is odd) map a vector valued modular form
	\[
	f = \sum_{\gamma \in A}\sum_{\substack{n \in \Z-Q(\gamma) \\ n \geq 0}}c(\gamma,n)e(n\tau)\e_{\gamma}
	\]
	in $M_{k,\rho_{A}^{*}}$ to the modular form
	\[
	\sum_{\gamma \in A}\sum_{\substack{n \in \Z-Q(\gamma) \\ n \geq 0}}b(\gamma,n)e(n\tau)\e_{\gamma},
	\]
	where
	\begin{align*}
	b(\gamma,n) &= \begin{dcases}
		c(r\gamma,pn) + p^{k-1}c(\gamma/r,n/p), & \text{if $m$ is even,} \\
		c(p\gamma,p^{2}n) + \epsilon_{p}^{\left(\frac{-1}{|A|}\right)-\sig(A)}\left( \frac{p}{|A|2^{-\sig(A)}}\right)p^{k-3/2}\left( \frac{-n}{p}\right)c(\gamma,n) \\
		+ p^{2k-2}c(\gamma/p,n/p^{2}), & \text{if $m$ is odd.}
		\end{dcases}	\end{align*}
		Here $\sig(A) = b^{+} -b^{-} \pmod 8$ is the signature of $A$, $\epsilon_{p}$ equals $1$ or $i$ according to whether $p$ is equivalent to $1$ or $-1$ modulo $4$ (the level of $L$ is automatically divisible by $4$ if the rank $m$ of $L$ is odd, hence $p$ is odd if $m$ is odd), and we understand that $c(\gamma/r,n/p) = 0$ if $p \nmid n$ and $c(\gamma/p,n/p^{2}) = 0$ if $p^{2} \nmid n$. Since $p$ is coprime to $N$, we can choose $\overline{p} \in \Z$ such that $p\overline{p} \equiv 1 \pmod N$, and then $\gamma = p\overline{p}\gamma$ in $A$, so $\gamma/p$ makes sense. Note the we have to apply the results of \cite{bruinierstein} to the lattice $(L,-Q)$ since we work with the dual Weil representation $\rho_{A}^{*}$, which amounts to changing $\sig(A)$ to $-\sig(A)$ in their formulas.
	
	\begin{Lemma}\label{lemma hecke}
		Let $\beta \in A$ with $Q(\beta) = 0 \pmod \Z$. Let $p$ be a prime which is coprime to the level of $L$. Then the Hecke operator $T_{r}(p)$ (if $m$ is even) or $T(p^{2})$ (if $m$ is odd) maps $E_{A,\beta}$ to
		\begin{align*}
		\begin{dcases}
		E_{A,\beta/r}+p^{k-1}E_{A,r\beta} , & \text{if $m$ is even,} \\
		E_{A,\beta/p}+p^{2k-2}E_{A,p\beta}, & \text{if $m$ is odd.}
		\end{dcases}
		\end{align*}
	\end{Lemma}
	
	\begin{proof}
		The Eisenstein series $E_{A,\beta}$ is the unique holomorphic modular form of weight $k$ for $\rho_{A}^{*}$ which has constant term $\e_{\beta} + (-1)^{\kappa}\e_{-\beta}$ and which is orthogonal to cusp forms with respect to the Petersson inner product. Since the Hecke operators preserve orthogonality to cusp forms, it suffices to check that the relation stated in the lemma on the constant coefficients. Using the formula for the action of the Hecke operators given above this is easy to verify.
	\end{proof}
	
	Multiplying by $\chi(\nu)$, summing over $\nu \pmod {N_{\beta}}$, and replacing $\nu$ by $p\nu$ or $\overline{p}\nu$ at the appropriate places (note that $N_{\beta}$ divides $N$, so $p$ is coprime to $N_{\beta}$), we obtain the formulas stated in Proposition~\ref{proposition hecke}.

\bibliography{references}
\bibliographystyle{alpha}

\end{document}